\documentclass[12pt]{amsart}
\usepackage{graphicx,amssymb,latexsym,bm,color,rotating}
\usepackage[dvips,centering,includehead,width=15.1cm,height=22.7cm]{geometry}
\usepackage{amsmath}
\usepackage{graphicx}
\usepackage{epsfig}
\usepackage{amssymb}
\usepackage[latin1]{inputenc}
\usepackage{multirow}
\usepackage{hyperref}

\input{xy} \xyoption{all}

\numberwithin{equation}{section}

\newtheorem{theorem}{Theorem}[section]
\newtheorem{lemma}[theorem]{Lemma}

\theoremstyle{definition} \newtheorem{definition}[theorem]{Definition}
\newtheorem{example}[theorem]{Example}
\newtheorem{remark}[theorem]{Remark}

 \newcommand{\F}{{\mathbb F}}

\newcommand{\D}{{\mathcal D}} \renewcommand{\P}{{\mathcal P}}

\newcommand{\dive}{{\text{div}}}
\newcommand{\inte}{{\text{int}}}

 \newcommand{\cone}{{\text{cone}}}
 \newcommand{\Ch}{{\text{Ch}}}

 \newcommand{\RR}{{\mathbb R}}
\newcommand{\ZZ}{{\mathbb Z}} \newcommand{\NN}{{\mathbb N}}
 
\newcommand{\CC}{{\mathbb C}}

\newcommand{\ww}{ {\bf w} }

\newcommand{\dd}{ {\bf d} }

\newcommand{\xx}{ {\bf x} }
\newcommand{\yy}{ {\bf y} }

\makeatletter
\newtheorem*{rep@theorem}{\rep@title}
\newcommand{\newreptheorem}[2]{%
\newenvironment{rep#1}[1]{%
 \def\rep@title{#2 \ref{##1}}%
 \begin{rep@theorem}}%
 {\end{rep@theorem}}}
\makeatother
\newreptheorem{theorem}{Theorem}
\newreptheorem{corollary}{Corollary}

\title[Double Gromov-Witten invariants of Hirzebruch surfaces]{The double Gromov-Witten invariants of Hirzebruch surfaces are piecewise polynomial}

\author{Federico Ardila \and Erwan Brugall\'e}
\date{\today}
\address{Federico Ardila. San Francisco State University, San Francisco, CA, USA. \newline Universidad de Los Andes, Bogot\'a, Colombia. University of California, Berkeley, CA, USA.}
\email{federico@sfsu.edu}
\address{Erwan Brugall\'e, Centre Mathématiques Laurent Schwartz, 91 128 Palaiseau Cedex, France.}  
\email{erwan.brugalle@math.cnrs.fr}
\thanks{The first author was partially supported by the United States National Science Foundation CAREER Award DMS-0956178.}

\begin{document}

\begin{abstract}
We define the \emph{double Gromov-Witten invariants}  of Hirzebruch
surfaces in analogy with double Hurwitz numbers, and we prove 
that  they 
satisfy a piecewise polynomiality property analogous to their
 1-dimensional counterpart. 
Furthermore
we show that each polynomial piece is either even or odd, and we compute its degree.
Our methods combine floor diagrams and Ehrhart theory.

\end{abstract}

\maketitle

\section{Introduction}
\label{sec:intro}
\emph{Hurwitz numbers} count the holomorphic maps $C\to \CC P^1$
of a fixed degree $d$, with prescribed ramification values, and prescribed
ramification profiles over each ramification value.
These numbers are connected to several areas of mathematics including
algebraic geometry, combinatorics, and representation theory, among others.
In particular, the ELSV formula  \cite{ELSV} relates 
\emph{simple} Hurwitz numbers (where there is a single
critical value with possibly less than $d-1$ preimages) to
the moduli spaces of complex algebraic curves $\overline{\mathcal M}_{g,n}$.

No generalization of the ELSV-formula is known yet for \emph{double}
Hurwitz numbers; however, 
as possible evidence toward such a generalization,
these numbers  enjoy a very
 rich structure. In particular, Goulden, Jackson, and Vakil proved,
 among other things,
that double Hurwitz numbers are piecewise polynomial 
 \cite{Vak3}.
Later on, 
 Cavalieri, Johnson, and Markwig used tropical geometry to give a new proof of  this piecewise
polynomiality 
\cite{CJM,CJM11}. In addition, they found a wall crossing
formula giving the difference of double Hurwitz numbers between two
adjacent chambers of polynomiality, generalizing the formula in genus
0 proved in \cite{SSV}.
An alternative approach to prove piecewise polynomiality and obtain these wall crossing formulas,
 based on the 
De Concini-Procesi-Vergne spaces of \cite{DPV1, DPV2},  was proposed by
the first author in \cite{Ard09}. 

\bigskip
\noindent \emph{Results.} 
In this note  we  introduce the
 \emph{double
  Gromov-Witten invariants of Hirzebruch surfaces}, which 
generalize double Hurwitz numbers, and we establish their piecewise
polynomiality.
We now make these assertions more
precise, referring
 to Section 
 \ref{sec:double hirz} for  precise definitions.

Given $k\ge 0$, we denote by $\F_k=\mathbb
P(\mathcal O_{\CC 
  P^1}(k)\oplus \mathcal O_{\CC P^1})$ the \textit{$k$th
  Hirzebruch surface}. We say that an algebraic curve in $\F_k$ is of
bidegree $(a,b)$ if it is linearly equivalent to the union of $a$
copies of  the section 
$\mathbb P(\mathcal O_{\CC P^1}(k)\oplus \{0\})$ and $b$ copies of
a fiber (see Section \ref{sec:hirz}).
The 
double Gromov-Witten invariants
of $\F_k$, 
which standard conventions denote 
$N^{\alpha,\beta,\widetilde \alpha,\widetilde
  \beta}_g(a,b,k)$,
count
algebraic
curves in $\F_k$ of a given bidegree $(a,b)$ and genus $g$, passing through an
appropriate configuration of points, and having fixed intersection
patterns
$\alpha,\beta,\widetilde \alpha,\widetilde\beta$
 with the sections 
$\mathbb P(\mathcal O_{\CC P^1}(k)\oplus \{0\})$ and
$\mathbb P(\{0\}\oplus \mathcal O_{\CC P^1})$. 

\medskip

It is more convenient for us to encode the double Gromov-Witten invariants of $\F_k$ in a function $F_{a,k,g}^{n_1,n_2}(\xx,\yy)$ as follows.
Let us fix $a>0$ and $k,g\ge 0$ as above, and let us also fix two additional 
non-negative integer numbers $n_1$ and $n_2$. 
We then define
$$\Lambda=\left\{(x_1,\ldots,x_{n_1},y_1,\ldots,y_{n_2}) \in \ZZ^{n_1}\times \ZZ^{n_2}
\ | \ \sum x_i +\sum y_j+ak=0\right\}\subset
\RR^{n_1}\times \RR^{n_2}.$$

Given an element $(\xx,\yy) = ((x_1,\ldots,x_{n_1}),(y_1,\ldots,y_{n_2})) \in \Lambda$, we store the multiplicities of the entries of $(\xx,\yy)$ in the 
four sequences $\alpha=(\alpha_i)_{i\ge 1},
\beta=(\beta_i)_{i\ge 1},\widetilde \alpha=(\widetilde 
\alpha_i)_{i\ge 1},$ and $\widetilde \beta=(\widetilde \beta_i)_{i\ge
  1}$ of non-negative integers, where $\alpha_i$ is the number of elements $x_j$
equal to 
$-i$,
$\beta_i$  is the number of elements $y_j$ equal to 
$-i$,
  $\widetilde \alpha_i$ is  the number of elements $x_j$ equal to 
$i$,
and  $\widetilde \beta_i$ is  the number of elements $y_j$ equal to 
$i$.
 Set
$b=\sum i(\widetilde \alpha_i +\widetilde\beta_i)$. In our examples, we will omit the parentheses in $\alpha,\widetilde
\alpha, \beta,$ and $\widetilde\beta$ 
to simplify the notation.
 For instance, the vector $\alpha = (2,0,1)$ will be denoted by $\alpha = 201$.

\begin{definition} The function $F_{a,k,g}^{n_1,n_2}(\xx,\yy)$ is defined by:
$$\begin{array}{cccc}
F_{a,k,g}^{n_1,n_2} : & \Lambda &\longrightarrow & \ZZ
\\ & (\xx,\yy) &\longmapsto & N^{\alpha,\beta,\widetilde \alpha,\widetilde
  \beta}_g(a,b,k)
\end{array}. $$
\end{definition}
\begin{example}
We have
\[
F_{3,2,1}^{4,5}((-2,- 2, -1,1), (-3, -1, -1,1, 2)) = N_1^{12, 201,1,11}(3,4,2)
\]
because the multiplicities in $(\xx,\yy)=((-2, -2, -1, 1), (-3, -1,
-1,1, 2))$ are given by 
$(\alpha,\beta,\widetilde\alpha,\widetilde\beta) = (12, 201,1, 11)$. 
Since the superscript of
$F_{3,2,1}$ denotes the sizes of the input vectors, we can drop it and
write 
$F_{3,2,1}\left((-2, -2, -1, 1), (-3, -1, -1, 1, 2)\right)$.
\end{example}

The following theorems are the main results of this note.

\begin{theorem}\label{th:main}
Let $k,g, n_1, n_2 \ge 0$ and $a \ge 1$ be 
fixed integers.
The function
\[
F_{a,k,g}^{n_1,n_2}(\xx,\yy)
\]
of double Gromov-Witten invariants of the Hirzebruch surface $\F_k$
is piecewise polynomial relative to the chambers of the 
hyperplane arrangement
\begin{eqnarray*}
\sum_{i \in S} x_i + \sum_{j \in T} y_j + kr= 0 &&  
(S \subseteq [n_1], \,\, T\subseteq [n_2], \,\, 0 \leq r \leq a), \\
y_i-y_j=0 &&  (1 \leq i < j \leq n_2)
\end{eqnarray*}
inside 
$\Lambda=\{(x_1,\ldots,x_{n_1},y_1,\ldots,y_{n_2}) \in \ZZ^{n_1}\times \ZZ^{n_2}
\ | \ \sum x_i +\sum
y_i+ak=0\}\subset
\RR^{n_1}\times \RR^{n_2}.$
\end{theorem}

\begin{theorem}\label{th:degree}
Each polynomial piece of $F_{a,k,g}^{n_1,n_2}(\xx,\yy)$ has degree
$n_2+3g+2a-2$, and is either even or odd. 
\end{theorem}

Theorems \ref{th:main} and \ref{th:degree} might suggest that a 2-dimensional generalization of the ELSV-formula could exist.

\bigskip

\noindent \emph{Techniques.} 
Our proofs of Theorems \ref{th:main} and \ref{th:degree}
 combine the enumeration of complex curves in complex surfaces
via \emph{floor diagrams},
together with  the approach proposed in \cite{Ard09, CJM11} to
study the piecewise polynomiality of double Hurwitz numbers.
Floor diagrams  were
introduced in {\cite{Br6b,Br7,Br6}}, and further explored in
{\cite{ArdBlo, Br8,Blo11, BlCoKe13,BGM,BlGo14,BlGo14bis,Br14,Br9,Bru14, FM, Liu13,LiuOss}}.
They allow  one to replace the geometric enumeration of curves by 
a purely combinatorial problem, applying the following general strategy. Suppose that one wants to enumerate algebraic curves in some complex
surface $X$ interpolating a configuration of points $\P$. Choose 
 a  non-singular rational curve $E$ in  $X$, and  degenerate $X$ 
into the union  of $X$ together with a
 chain of copies of the compactified normal bundle $\mathcal N_E$ 
of $E$ in $X$; moreover, 
specialize exactly one  point of $\P$ to
 each of these copies of $\mathcal N_E$. Now
floor diagrams encode the limit of curves
under enumeration in this degeneration process.
In good situations, including the one we deal with here,
all limit curves  can be
completely 
recovered only from the combinatorics of the 
floor
diagrams. We refer to \cite[Section 1.1]{Bru14} for more details about
the heuristic of the floor decomposition technique,
as well as to \cite{IP00,LiRu01,Li02} for degeneration formulas in
enumerative geometry.

It turns out that the tropical  count of double Hurwitz numbers performed in \cite{CJM} can be
interpreted as a floor diagram count in dimension 1.
The underlying combinatorial objects are very similar, and thus   
we are able to transpose part of the approach
from \cite{Ard09, CJM11}  to the 2-dimensional case. The key idea is to interpret the relevant 
combinatorial problem as the (weighted) enumeration of lattice points in flow polytopes, and apply techniques from Ehrhart theory.

\bigskip

\noindent \emph{Organization.} 
The paper is organized as follows. Double Gromov-Witten invariants of Hirzebruch surfaces are defined in
Section \ref{sec:double hirz}, and we explain how to compute them via floor
diagrams in Section \ref{sec:floordiagrams}. This reduces our enumerative geometric question to a combinatorial question, which we treat in the remaining sections. 
Section
\ref{sec:partition} recalls and extends some facts from 
 Ehrhart theory. 
We use these results in Section \ref{sec:proof}, where we rephrase the enumeration of floor diagrams in terms of polyhedral geometry, and complete the proof of 
our main results. 
We work out a concrete
example in Section \ref{sec:example}, and 
end the paper in
Section
\ref{sec:conclusion} with some concluding remarks about possible
extensions of this work.

\section{Double Gromov-Witten invariants of Hirzebruch surfaces}

In this section, to state our results, we will require some
familiarity with the geometry of complex curves;
 for an introduction, see \cite{GriHar78,Beau}. For more details about the
 enumerative geometry of Hirzebruch surfaces, we refer to \cite{Vak2}.

\label{sec:double hirz}
\subsection{Hirzebruch surfaces}\label{sec:hirz}
Recall the $k$th
  Hirzebruch surface, with $k\ge 0$, is denoted by
 $\F_k$, i.e. 
 $\F_k=\mathbb
P(\mathcal O_{\CC 
  P^1}(k)\oplus \mathcal O_{\CC P^1})$.
Any compact complex surface  admitting a holomorphic fibration
to $\CC P^1$ with fiber
 $\CC P^1$ is isomorphic to exactly one of the Hirzebruch surfaces.

For example one has $\F_0=\CC P^1\times \CC P^1$.
The surface $\F_1$ is the projective plane blown up at a point,
and $\F_2$ is the quadratic cone with equation $x^2+y^2+z^2$ in
$\CC P^3$ blown up at the node. In the last two cases, the fibration
 is given by the 
extension  of the projection from the blown-up point to a 
line (if
$k=1$) or a hyperplane section (if $k=2$) 
that does not pass through the blown-up point.

Let us denote by $B_k$ (resp. $E_k$ and $F_k$) the section 
$\mathbb P(\mathcal O_{\CC P^1}(k)\oplus \{0\})$ (resp. the section
$\mathbb P(\{0\}\oplus \mathcal O_{\CC P^1})$ and a fiber). 
The curves $B_k$, $E_k$, and $F_k$ have self-intersections
$B_k^2=k$, $E_k^2=-k$, and  $F_k^2=0$. When $k\ge 1$, the curve $E_k$ itself
determines uniquely the 
Hirzebruch surface, since it is the only reduced and irreducible
algebraic curve in $\F_k$ with negative self-intersection.
The group 
$\mbox{Pic}(\F_k)=\mbox{H}_2(\F_k,\mathbb Z)$
is isomorphic to $\mathbb Z\times\mathbb Z$ and is generated by the
classes of $B_k$ and  $F_k$. Note that we have
$E_k=B_k-kF_k$ in $\mbox{H}_2(\F_k,\mathbb Z).$ 
An algebraic curve $C$
in $\F_k$ is 
said to be of \textit{bidegree} $(a,b)$ if it realizes the homology
class $aB_k+bF_k$ in $\mbox{H}_2(\F_k,\mathbb Z)$.
By the adjunction formula, a non-singular algebraic curve $C$ of bidegree 
$(a,b)$ in $\F_k$ has genus 
$$g(C)=\frac{a(a-1)}{2}k +ab -a -b +1 .$$

\subsection{Double Gromov-Witten invariants}\label{sec:double GW}
Let us fix four  
integers\footnote{As we will see in Remark \ref{rem:a=0}, one can
  extend our definition of Gromov-Witten invariants to the case $a=0$
  with some extra care.}
$a>0$, and $b,k,g\ge 0$, as well as
four sequences of non-negative
integers
  $\alpha=(\alpha_i)_{i\ge 1}$, $\widetilde \alpha
=(\widetilde \alpha_i)_{i\ge 1} $, $\beta=(\beta_i)_{i\ge 1}$, $\widetilde \beta
=(\widetilde \beta_i)_{i\ge 1} $
  such that
$$\sum_i i(\alpha_i+\beta_i)=ak+b  \quad \text{and}\quad 
\sum_i i(\widetilde \alpha_i+\widetilde \beta_i)=b. $$
In particular
this implies that only finitely many terms
of the four sequences are non-zero. We define
$l=2a+g+\sum (\beta_i+\widetilde\beta_i ) -1$. 

Next, let us choose a generic configuration 
$$\omega=\{q^1_1,\ldots, q^1_{\alpha_1},\ldots,q^i_1,\ldots, q^i_{\alpha_i},\ldots,
 p_1,\ldots, 
p_{l}, \widetilde q^1_1,\ldots,
\widetilde q^1_{\alpha_1},\ldots,\widetilde q^i_1,\ldots, 
\widetilde q^i_{\alpha_i},\ldots\}$$
 of $l+\sum\left(\alpha_i+\widetilde\alpha_i \right)$ points in $\F_k$ such that
$q^i_j\in B_k$, $\widetilde q^i_j\in E_k$, and 
$p_i\in \F_k\setminus(B_k\cup E_k)$.

We denote by $N^{\alpha,\beta,\widetilde \alpha,\widetilde
  \beta}_g(a,b,k)$ the number of irreducible
 complex algebraic curves $C$ in $\F_k$ of
genus $g$ such that
\begin{enumerate}
\item $C$ passes through all the points $q^i_j$, $\widetilde q^i_j$, and $p_i$;
\item $C$ has order of contact $i$ with $B_k$ at $q^i_j$, and
 has $\beta_i$ other (non-prescribed) points with order of contact
  $i$ with $B_k$;
\item $C$ has order of contact $i$ with $E_k$ at $\widetilde q^i_j$,
  and
has $\widetilde \beta_i$ other (non-prescribed) points with order of contact
  $i$ with $E_k$.
\end{enumerate}

This number is finite and doesn't depend on the chosen generic
configuration of points. We call this number a \emph{double
  Gromov-Witten invariant\footnote{These numbers are also called
    \emph{Gromov-Witten invariants of $\F_k$ relative to $B_k\cup
      E_k$} and have an intersection theoretic interpretation in some
    suitable moduli spaces, see for example \cite{IP00}.} of $\F_k$} in analogy with double Hurwitz
numbers.

\begin{example}
Recall the we  omit the parentheses in $\alpha,\widetilde
\alpha, \beta,$ and $\widetilde\beta$ 
to simplify the notation.
By the adjunction formula, one has $N^{\alpha,\beta,\widetilde \alpha,\widetilde
  \beta}_g(a,b,k)=0$ as soon as $g> \frac{a(a-1)}{2}k +ab -a -b +1 $.

It is well known that there exist exactly 2 conics in $\CC P^2$
through 4 points in generic positions and tangent to a fixed
line.
In our notation, this translates to 
$N^{0,01,0,0}_0(2,0,1)=2$.

If we now look at conics 
through 3 points in generic position
in $\CC P^2$ and tangent to a fixed line at a prescribed point, then
we find exactly one such conic, i.e. $N^{01,0,0,0}_0(2,0,1)=1$.

As  less straightforward examples, we give the values
$N^{01,0,0,01}_0(2,2,0)=8$ and 
$N^{01,1,0,1}_0(3,1,1)=8$. We will compute these numbers in the
next section, using floor diagrams.
\end{example}

\begin{remark}\label{rem:a=0}
Theorem \ref{th:main} extends trivially to the case $a=0$. 
However, for simplicity we chose to define 
the invariants $N^{\alpha,\beta,\widetilde \alpha,\widetilde
  \beta}_g(a,b,k)$ by counting immersed algebraic curves instead of
maps. As a consequence, the cases $a=0$ and $b>1$ are problematic with
our simplified definition because of the appearance of
 non-reduced curves. However with the suitable definition of
 Gromov-Witten invariants in terms of maps, we have (see
\cite[Section 8]{Vak2})
$$N^{u_b,0,0,u_b}_0(0,b,k)=N^{0,u_b,u_b,0}_0(0,b,k)=\frac1b \quad
\mbox{and}\quad
N^{0,u_b,0,u_b}_0(0,b,k)=1,$$
where $u_b$ is the $b-th$ vector of the canonical basis of $\RR^n$,
and 
$$N^{\alpha,\beta,\widetilde \alpha,\widetilde
  \beta}_g(0,b,k)=0$$
in all the other cases. In particular, the conclusion of Theorem
\ref{th:main} holds when $a=0$ and either $g\ne 0 $ or $n_1=0$. Note
that in the case when $a=g=0$ and $n_1=1$, the function 
\[
F_{0,k,0}^{1,1}(\pm b,\mp b)=\frac1b
\]
is rational instead of polynomial.
\end{remark}

\section{Floor diagrams}
\label{sec:floordiagrams}

Here we recall how to enumerate complex curves in $\F_k$ using
floor diagrams. 
We use
notation inspired by {\cite{FM,ArdBlo}}.

\begin{definition}\label{def:fd}
A \emph{marked floor diagram} $\D$ for $\F_k$  consists of

\begin{enumerate}

\item
 A vertex set $V = L \sqcup C \sqcup R$ where $C$ is totally ordered
 from left to right,  $L=\{q_1,\ldots,q_{l}\}$ is unordered and to
 the left of $C$, and 
$R=\{\widetilde q_1,\ldots,\widetilde q_{r}\}$ is unordered and to the right of $C$.

\item
 A coloring of the vertices with black, white, and gray, so that every vertex in $L$ and $R$ is white.

\item
A set $E$ of edges, directed from left to right, such that 

\noindent $\bullet$ The resulting graph is connected.

\noindent $\bullet$ Every white vertex 
is incident to
exactly one edge, which connects it to a black vertex.

\noindent  $\bullet$ Every gray vertex 
is incident to
exactly two edges; one coming from a black vertex, and the other one going to a black vertex.
 
 \item
A choice of positive integer weights $w(e)$ on each edge $e$ such that if we define the \emph{divergence} of a vertex $v$ to be
\begin{displaymath}
\dive(v) := \sum_{ \tiny
     \begin{array}{c}
  \text{edges }e\\
v' \stackrel{e}{\to} v
     \end{array}
} w(e) -   \sum_{ \tiny
     \begin{array}{c}
  \text{edges }e\\
v \stackrel{e}{\to} v'
     \end{array}
} w(e) . 
\end{displaymath}
then

 $\bullet$ $\dive(v) = k$ for every black vertex $v$.
 
 $\bullet$ $\dive(v) = 0$ for every gray vertex $v$.

\end{enumerate}

\begin{example}
For didactical reasons, we   
interject an example in this long series of definitions.
Figure \ref{fig:floordiagram1} shows a floor diagram for $\F_2$. We draw dotted lines to separate $L$, $C$, and $R$. 
To simplify our pictures, we omit the labels of the vertices in $L$
and $R$. For instance, the picture of Figure \ref{fig:floordiagram1}
 actually represents
three different floor diagrams, corresponding to the three
different ways of assigning 
the labels $q_1,q_2$, and $q_3$ to the vertices in $L$. 
When an edge $e$ has weight $w(e)>1$, we write that weight next to it. Since there is no risk of confusion, we omit the (left-to-right) edge directions. Notice that $\dive(v)=2$ for every black vertex and $\dive(v) = 0$ for all the gray vertices. 
\end{example}
 
\begin{center}
\begin{figure}[h]
\begin{tabular}{ccc}
 \includegraphics[height=2cm]{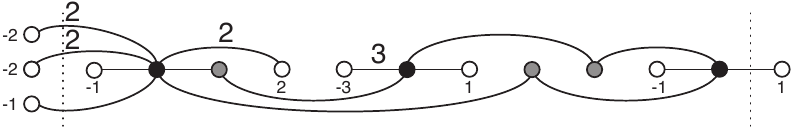}
\end{tabular}
\caption{A floor diagram for the Hirzebruch surface $\F_2$.}
\label{fig:floordiagram1}
\end{figure}
\end{center}

We associate several parameters to a marked floor diagram $\D$:

\smallskip

\noindent
$\circ$ 
The \emph{type} of $\D$ is the pair $(n_1, n_2)$ where 
$\D$ has $n_1$
white vertices in $L\cup R$ (i.e. $n_1=l+r$), 
and $n_2$ white vertices in $C$.

\noindent 
$\circ$ 
 The \emph{divergence sequence} is a vector 
$(\xx, \yy) \in \ZZ^{n_1} \times \ZZ^{n_2}$ of length $n_1 + n_2$ where

$\bullet$ $\xx=(\dive(q_1),\ldots,\dive(q_{l}),
\dive(\widetilde q_1),\ldots, \dive(\widetilde q_{r}))$ is the sequence of 
divergences  of 
vertices in $L$ and $R$;

$\bullet$ $\yy$ is the sequence of  divergences of white vertices in $C$, listed from left to right.

\noindent
Since the sum of all the divergences in the graph must be $0$, 
we must have
\[
\sum x_i + \sum y_j = -ka
\]
where $a$ is the number of black vertices of $\D$.
The vector $\xx$ is called the \emph{left-right sequence} of $\D$.

\item $\circ$
The \emph{divergence multiplicity vector} is a sequence of four vectors $(\alpha(\xx), \beta(\yy), \widetilde \alpha(\xx), \widetilde \beta(\yy))$ where

 $\bullet$  $\alpha_i$ is the number of (white) vertices $v$ in $L$ with $\dive(v) = -i$;

 $\bullet$  $\widetilde \alpha_i$ is the number of (white) vertices $v$ in $R$ with $\dive(v) = i$;

 $\bullet$  $\beta_i$ is the number of white vertices $v$ in $C$ with $\dive(v) = -i$;

 $\bullet$  $\widetilde \beta_i$ is the number of white vertices $w$ in $C$ with $\dive(v) = i$.

\noindent Clearly $(\xx, \yy)$ determines $(\alpha, \beta, \widetilde \alpha, \widetilde \beta)$ completely, while $(\alpha, \beta, \widetilde \alpha, \widetilde \beta)$ determines $(\xx,\yy)$ up to the order of the coordinates in $\xx$ and $\yy$.

\noindent $\circ$ 
The \emph{bidegree} of $\D$ is the pair $(a,b)$ of positive integers such that 
\[
\sum_i i(\widetilde \alpha_i + \widetilde \beta_i) = b, \qquad \sum_i i(\alpha_i + \beta_i) = ka+b.
\]
Recall that $a$ is the number of black vertices of $\D$.

\noindent $\circ$ 
The \emph{genus} $g(\D)$ of $\D$ is its first Betti number; it equals  $g(\D)=1 - |V| + |E|$.

\noindent 
$\circ$ 
The \emph{multiplicity}  $\mu(\D)$ is the product of the internal edge weights, where an edge is \emph{internal} if it connects two vertices of $C$.

\end{definition}

\begin{example}
Suppose that the vertices of the floor diagram for
$\F_2$ in Figure \ref{fig:floordiagram1} are, from the top to the
bottom, $q_1, q_2$, and $q_3$. Then
its divergence sequence
is
\[
((-2,-2,-1,1),(-1,2,-3,1,-1)).
\]
The divergences of the white vertices in $L$ and $R$ are $-2,-2,-1$ and $1$; they are encoded, respectively, in the vectors $\alpha=12,  \, \widetilde \alpha=1.$
The divergences of the white vertices in $C$ are $-1, -1, -3$, and $1,2$; they are respectively encoded in the vectors $\beta=201,  \, \widetilde\beta=11$.
Therefore the divergence multiplicity vector is
\[
(\alpha,\widetilde \alpha, \beta, \widetilde\beta) = (12, 1, 201, 11).
\]
The negative white divergences add up to $-\sum_i i(\alpha_i+\beta_i)
= -10 = -(2a+b)$ and the positive white divergences add up to $\sum_i
i(\widetilde\alpha_i + \widetilde\beta_i) = 4 = b$. Therefore $\D$ has
bidegree $(a,b) = (3,4)$. Visibly, the genus is $g(\D)=1$ and the
multiplicity is $\mu(\D) = 6$.
\end{example}

In Definition \ref{def:fd}, we use the adjective \emph{marked} in
reference to the corresponding objects in {\cite{Br7,FM,ArdBlo}},
which are \emph{floor diagrams} endowed with additional
structure. However, since we do not consider unmarked floor diagram in this
note, we will abbreviate \emph{marked floor diagram} to  \emph{floor
  diagram} in the rest of the text.

The following theorem is a very minor variation on {\cite[Theorem 3.6]{Br6b}}.
It replaces
the enumeration of algebraic curves by the enumeration of floor diagrams.
\begin{theorem}\label{th:BM}
Let $a>0$ and $b,k,g\ge 0$ be
   four  
integer numbers, and $\xx$ a vector with coordinates in $\ZZ\setminus\{0\}$. 
We write $\alpha(\xx)=\alpha$ and
$\widetilde\alpha(\xx)=\widetilde\alpha$.
Then for any two sequences of
non-negative
integer numbers
$\beta=(\beta_i)_{i\ge 1}$ and $\widetilde \beta
=(\widetilde \beta_i)_{i\ge 1} $
  such that
$$\sum_i i (\alpha_i+\beta_i)=ak+b,  \quad 
\text{and}\quad 
\sum_i i (\widetilde \alpha_i+\widetilde \beta_i)=b, \quad
$$
one has
\begin{displaymath}
 N^{\alpha,\beta,\widetilde \alpha,\widetilde
  \beta}_g(a,b,k)= \sum_\D \mu(\D) ,
\end{displaymath}
where the sum runs over all  floor diagrams $\D$  of bidegree $(a,b)$,
genus $g$, left-right sequence $\xx$, and
divergence multiplicity vector
$(\alpha,\beta,\widetilde\alpha,\widetilde\beta)$  for $\F_k$.
\end{theorem}

\begin{proof}
Strictly speaking, the tropical proof of {\cite[Theorem 3.6]{Br6b}} uses 
Mikhalkin's Correspondence Theorem {\cite[Theorem 1]{Mik1}}, and 
proves our
theorem only when 
$$\alpha=\widetilde \alpha=0,\quad \beta=(ka+b,0,\ldots,0) ,
\quad\mbox{and}\quad 
\widetilde\beta= (b,0,\ldots,0).$$
A generalization of Mikhalkin's Correspondence Theorem which
covers the case of
 curves satisfying tangency conditions with toric divisors
can be found for example in  {\cite[Theorem 2]{Shu9}}. 
In particular, the generalization of the  proof of
{\cite[Theorem 3.6]{Br6b}} to our situation is straightforward.
Alternatively, a proof  of Theorem \ref{th:BM} within classical algebraic
geometry can be obtained by a straightforward adaptation of
{\cite[Section 5.2]{Bru14}}.

For the reader's convenience, we briefly recall the correspondence
between  floor diagrams and  algebraic curves following
 \cite{Bru14}.
We use
notation from Section \ref{sec:double GW}.
As explained in the introduction, we may degenerate $\F_k$ into a
chain   $Z=\bigcup_{i=0}^{l+1}Y_i$ 
of $l+2$ surfaces all isomorphic to $\F_k$, where two consecutive
components $Y_i$ and $Y_{i+1}$
intersect along $E_k$ in $Y_i$  and $B_k$ in $Y_{i+1}$. 
In this degeneration of $\F_k$, 
we specialize  the  point  $p_i$ of $\omega$ to $Y_i$, the points
$q_i^j$ to $Y_0$, and the points $\widetilde q_i^j$ to $Y_{l+1}$.
Let $X$ be a limit in $Z$ of curves under enumeration.
It turns out that 
any irreducible component of  $X$  has either
bidegree $(1,b_i)$ or $(0,b_i)$, and 
 we associate a floor diagram to $X$ as follows: 
\begin{itemize}
\item black vertices correspond to points $p_i$ lying on an irreducible
component of $X$  of bidegree $(1,b_i)$; 
\item grey vertices 
correspond to points $p_i$ lying on a chain of irreducible
components of $X$ of bidegree $(0,b_i)$ 
 connecting two
irreducible components of $X$ of bidegree $(1,b_r)$ and $(1,b_s)$; the weight of the two
corresponding edges is $b_i$;
\item white vertices not in $L\cup R$ 
correspond to points $p_i$ lying on a chain of irreducible
components of $X$ of bidegree $(0,b_i)$ connecting an
irreducible component of $X$ of bidegree $(1,b_r)$ to either 
$B_k$ in $Y_0$ or $E_k$ in $Y_{l+1}$; the weight of the 
corresponding edge is $b_i$;
\item  white
vertices in $L$ (resp. $R$) correspond to points $q_i^j$
(resp. $\widetilde q_i^j$). 
\end{itemize}
The multiplicity of the obtained floor diagram is precisely the
number of curves under enumeration that degenerate to $X$ when $\F_k$
degenerates to $Z$.
\end{proof}

\begin{example}
Figure \ref{fig:floordiagram2}a represents the only
 floor diagram in $\F_1$ of
bidegree $(2,0)$, genus 0, and 
divergence multiplicity vector
$(0,01,0,0)$. 
Hence
$N^{0,01,0,0}_0(2,0,1)=2$
according to Theorem
\ref{th:BM}.
\end{example}

\begin{example}
Figure \ref{fig:floordiagram2}b represents the only
 floor diagram in $\F_1$ of
bidegree $(2,0)$, genus 0, and 
divergence multiplicity vector
$(01,0,0,0)$. 
By Theorem
\ref{th:BM}, we have
$N^{01,0,0,0}_0(2,0,1)=1$.
\end{example}

\begin{example}
Figure \ref{fig:floordiagram2}c represents the only
 floor diagram in $\F_0$ of
bidegree $(2,2)$, genus 0, and 
divergence multiplicity vector
$(01,0,0,01)$. 
Therefore by Theorem
\ref{th:BM}, we have 
$N_0^{01,0,0,01}(2,2,0)=8$.
\end{example}

\begin{center}
\begin{figure}[h]
\begin{tabular}{ccccc}
 \includegraphics[height=1.2cm]{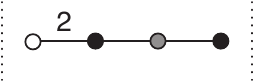} & \hspace{6ex} &
 \includegraphics[height=1.2cm]{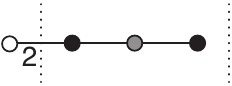} &  \hspace{6ex} &
\includegraphics[height=1.2cm]{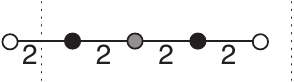}
\\ a) $\mu(\D)=2$ &&b) $\mu(\D)=1$ && c)  $\mu(\D)=8$
\end{tabular}
\caption{Computation of  $N^{0,01,0,0}_0(2,0,1)=2$, $N^{01,0,0,0}_0(2,0,1)=1$ and $N^{01,0,0,01}_0(2,2,0)=8$.}
\label{fig:floordiagram2}
\end{figure}
\end{center}

\begin{example}
Figure \ref{fig:floordiagram3}
represents all
 floor diagrams in $\F_1$ of
bidegree $(3,1)$, genus 0, and 
divergence multiplicity vector
$(01,1,1,0)$. 
By Theorem
\ref{th:BM}, we have
$$N^{(01,1,1,0)}_0(3,1,1)=1+1+1+4+1=8.$$
\end{example}

\begin{center}
\begin{figure}[h]
\begin{tabular}{ccccc}
 \includegraphics[height=1.1cm]{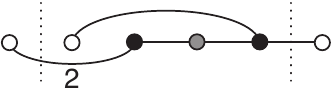} & \hspace{1ex} &
 \includegraphics[height=1.1cm]{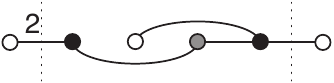} & \hspace{1ex} &
\includegraphics[height=1.1cm]{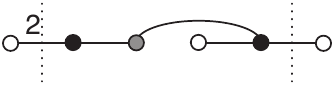}
\\  a) $\mu(\D) = 1$ && b) $\mu(\D)=1$ && c)  $\mu(\D)=1$
\end{tabular}
\begin{tabular}{ccc}
\\ \\ \includegraphics[height=1.1cm]{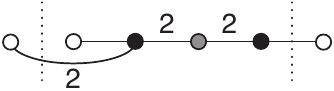} & \hspace{6ex} &
\includegraphics[height=1.1cm]{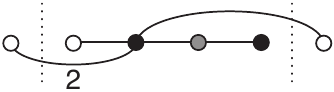}
\\ d) $\mu(\D)=4$ && e)  $\mu(\D)=1$
\end{tabular}
\caption{Computation of $N^{(01,1,1,0)}_0(3,1,1)=8$.}
\label{fig:floordiagram3}
\end{figure}
\end{center}

\section{Weighted partition functions and weighted Ehrhart reciprocity}
\label{sec:partition}
In this section we collect a few results about partition functions and their weighted counterparts that will be useful to us in what follows.

Let $X=\{a_1, \ldots, a_m\} \subset \ZZ^d$ be a finite multiset 
of lattice vectors in $\RR^d$. The \emph{rank} $r(X)$ of $X$ is the dimension of the real span of $X$. We may regard $X$ as an $m \times d$ matrix whose columns are $a_1, \ldots, a_m$. We say $X$ is \emph{unimodular} if all the maximal minors are equal to $-1, 0,$ or $1$.

The \emph{cone} of $X$ is $\cone(X) = \{\sum t_ia_i\, | \, t_i \geq
0\}$. We will assume that $X$  is \emph{pointed}; \emph{i.e.}, $\cone(X)$ does not contain a nontrivial linear subspace. This is equivalent to requiring that $X$ lies in some open half-space of $\RR^d$.

\subsection{Weighted partition functions}
If $X=\{a_1, \ldots, a_m\} \subset \ZZ^d$ is a pointed vector configuration, we define the \emph{partition function} $\P_X: \ZZ^d \rightarrow \ZZ$ to be
\[
\P_X(c) = ( \textrm{number of ways of writing } c = \sum_{i=1}^m c_ia_i \textrm{ with } c_i \in \NN).
\]
Equivalently, $\P_X(c)$ is the number of integer points in the polytope:
\[
P_X(c) = \{z \in \RR^m \, | \, Xz=c, z \geq 0\}.
\]

More generally, if 
$f \in \RR[z_1, \ldots, z_m]$ 
is a polynomial function, we define the \emph{weighted partition function}
\[
\P_{X,f}(c) = \sum_{z \in P_X(c) \cap \ZZ^m} f(z).
\]
We may think of this as a ``discrete integral" of the function $f$ over the polytope $P_X(c)$.
When $f=1$ we recover the ordinary partition function.

We wish to know how the polytope $P_X(c)$ and the weighted partition function $\P_{X,f}(c)$ vary for fixed $X$ and variable $c$. As $c$ varies, the defining hyperplanes of the polytope move, but their facet directions stay fixed. Naturally, as we push hyperplanes past vertices, we may change the combinatorial shape of the polytope. However, when the combinatorial shape is fixed, it is reasonable to expect that the discrete integral $\P_{X,f}(c)$ of a polynomial $f$ should change predictably. We now make this precise.

The \emph{chamber complex} $\Ch(X)$ of $X$ is a polyhedral complex supported on $\cone(X)$. It is given by the common refinement of all the cones spanned by subsets of $X$.

A function $f: \ZZ^n \rightarrow \RR$ is \emph{quasipolynomial} if there exists a sublattice $\Lambda \subseteq \ZZ^n$ of full rank and polynomials $f_1, \ldots, f_N$ corresponding to the different cosets $\Lambda_1, \ldots, \Lambda_N$ of $\Lambda$ such that $f(v) = f_i(v)$ for every $v \in \Lambda_i$. 
A function $f: \ZZ^n \rightarrow \RR$ is \emph{piecewise
  quasipolynomial 
relative to $\Ch(X)$} if the restriction of $f$ to any given face $F$ of the chamber complex $\Ch(X)$ is equal to a quasipolynomial function $f^F$ depending on $F$.

\begin{theorem}\cite{Bla, Stu95}\label{th:sturmfels}
For any pointed vector configuration $X \subset \ZZ^d$, the partition
function $\P_{X}$  is piecewise quasipolynomial 
relative to 
the chamber complex $\Ch(X)$. Furthermore, if $X$ is unimodular, then $\P_X$ is piecewise polynomial. The polynomial pieces of $\P_X$ have degree $|X| - r(X)$.
\end{theorem}

We will need an extension of this result. For each subset $Y \subseteq
X$ let $\pi_Y: \RR^m \rightarrow \RR$ be the function 
$\pi_Y(z_1,\ldots,z_m) = \prod_{i \in Y} z_i$. 
(Here we are identifying the coordinates of $\RR^m$ with the corresponding vectors of $X=\{a_1, \ldots, a_m\}$.)

\begin{theorem}\label{th:pp}
For any pointed vector configuration $X \subset \ZZ^d$ and any subset
$Y \subseteq X$, the weighted partition function $\P_{X, \pi_Y}$  is
piecewise quasipolynomial 
relative to
the chamber complex $\Ch(X)$. Furthermore, if $X$ is unimodular, then $\P_{X, \pi_Y}$ is piecewise polynomial. The polynomial pieces of $\P_X$ have degree $|X|+|Y| - r(X)$.
\end{theorem}

\begin{remark}
In fact, Theorem \ref{th:pp} holds for any weighted partition function $\P_{X, f}(c)$ where $f$ is polynomial. This result is known to experts on partition functions, and certainly not surprising in view of Theorem \ref{th:sturmfels}. We give a proof of the special case we need. This proof can be adapted to the general case; 
for details, see \cite{Ard09}.
\end{remark}

\begin{proof}
Consider the multiset $X \cup Y$ obtained by adding to $X$ a second copy of the vectors in $Y$. The partition function $\P_{X \cup Y}(c)$ of  this new configuration $X \cup Y$ counts the ways of writing $c = \sum_{i \in X} c_ia_i + \sum_{i \in Y} d_ia_i$ with $c_i, d_i \in \NN$.

Now note that this is the same as first writing $c = \sum_{i \in X} k_ia_i$ with $k_i \in \NN$, and then writing $k_i = c_i + d_i$ with $c_i, d_i \in \NN$ for each $i \in Y$. For each choice of $(k_1, \ldots, k_m)$ there are $\prod_{i \in Y} (k_i+1)$ of carrying out the second step. Therefore
\[
\P_{X \cup Y}(c) = \sum_{k \in \P_X(c) \cap \ZZ^m} \prod_{i \in Y} (k_i+1) = \sum_{T \subseteq Y} \P_{X, \pi_T}(c).
\]
By the inclusion-exclusion formula, we get
\begin{equation} \label{eq:pp}
\P_{X, \pi_Y}(c) =  \sum_{T \subseteq Y} (-1)^{|Y-T|} \P_{X \cup T}(c). 
\end{equation}
For any $T$, the partition function $\P_{X \cup T}(c)$ is quasipolynomial on each face of the chamber complex of $X \cup T$, which coincides with the chamber complex of $X$. It follows that the weighted partition function $\P_{X, \pi_Y}(c)$  is quasipolynomial on $\Ch(X)$ as well. The second statement also follows immediately.

Finally, the claim about the degree of $\P_{X, \pi_Y}(c)$ follows
immediately from (\ref{eq:pp}). However, it is also useful to give a
more intuitive argument. Regard $\P_{X, \pi_Y}(c)$ as the
\emph{discrete integral} of the function $\pi_Y$ (which is polynomial
of degree $|Y|$) over the polytope $P_X(c)$ (which has dimension $d =
|X| - r(X)$). In each face of the chamber complex, where $P_X(c)$ has
a fixed combinatorial shape, the actual integral $\int_{P_X(c)}
\pi_Y(k) dk$ is   polynomial in $c$ of degree $d+|Y|$ 
by
\cite[Theorem 2.15]{BrVe}.
Now, we can approximate this integral using increasingly fine lattices; it equals
\[
 \lim_{N \rightarrow \infty} \frac1{N^d}
 \sum_{k \in \P_X(c) \cap \left(\frac1N \ZZ\right)^m} \pi_Y(k) =
 \lim_{N \rightarrow \infty}
 \sum_{l \in \P_X(Nc) \cap \ZZ^m} \frac{\pi_Y(l)}{N^{d+|Y|}} =
\lim_{N \rightarrow \infty} \frac{\P_{X, \pi_Y}(Nc)}{N^{d+|Y|}}.
\]
This is only possible if $P_{X, \pi_Y}(c)$ also has degree $d+|Y| = |X|+|Y|-r(X)$, as desired.
\end{proof}

\begin{example}
The motivating example is \emph{Kostant's partition function}
which is defined to be the partition function
$\P_{A_{d-1}}(c)$ of the root system $A_{d-1} = \{e_i -
e_j \, : \, 1 \leq i < j \leq d\}$ 
 where $e_1, \ldots, e_d$ is the canonical
basis of $\RR^d$. 
This function plays a fundamental role in the representation theory of the Lie algebra $\mathfrak{gl}_n$. (More generally, the representation theory of a semisimple Lie algebra  is intimately related to the partition function of the corresponding root system; see \cite{FulHar, Kos} for details.)

It is well known and not difficult to show that $A_{d-1}$ is unimodular. 
In the vector space  $V_d = \{x \in \RR^d \, : \, x_1 + \cdots + x_d = 0\}$, consider the \emph{all-subset hyperplane arrangement} (also known as \emph{discriminant arrangement} ${\mathcal S}_d$ consisting of the following $2^{d-1}-1$ distinct hyperplanes. 
\[
{\mathcal S}_d: \qquad \sum_{i \in S} x_i = 0 \qquad  \qquad (\emptyset \subsetneq S \subsetneq [d]).
\]
Note that $ \sum_{i \in S} x_i = 0$ and $ \sum_{i \in [d]-S} x_i = 0$ are the same
hyperplane. 
The root system $A_{d-1}$ is contained in $V_d$, and the hyperplanes
spanned by roots in $A_{d-1}$ are precisely the hyperplanes in ${\mathcal S}_d$. Therefore the
chamber complex of $A_{d-1}$ in $V_d$ is the restriction of the all-subset arrangement ${\mathcal S}_d$ to $\cone(A_{d-1})$.
\end{example}

\subsection{Weighted Ehrhart reciprocity}


Say a polytope $P \subset \RR^d$ is \emph{rational} if its vertices are rational points, and \emph{integral} if its vertices are lattice points. Let $P^\circ$ be the \emph{relative interior} of $P$; that is, the topological interior of $P$ inside its affine span.

\begin{theorem}[Ehrhart reciprocity \cite{BecRob, Ehr}]\label{th:Ehrh}
  Let $P$ be a rational polytope in $\RR^m$. For each positive integer $n$, let 
\[
L_P(n) = |nP \cap \ZZ^m|, \qquad L_{P^\circ}(n) = |nP^\circ \cap \ZZ^m|.
\]
count the lattice points in the $n$th dilate of $P$ and in its interior, respectively. 
Then $L_P$ and $L_{P^\circ}$ extend to quasipolynomial functions which satisfy
\[
L_{P^\circ}(x) = (-1)^{\dim P} L_P(-x).
\]
Furthermore, if $P$ is a lattice polytope, then $L_P$ and $L_{P^\circ}$ are polynomial.
\end{theorem}

The function $L_P(x)$ is called the \emph{Ehrhart (quasi-)polynomial} of $P$.
We need a weighted version of this result.

\begin{theorem}[Weighted Ehrhart reciprocity]\label{th:Ehrhart} Let $P$ be a rational polytope in $\RR^m$ and $f:\RR^m\to\RR$ be a polynomial function.
 For each positive integer $n$, let 
\[
L_{P,f}(n) = \sum_{z \in nP \cap \ZZ^m} f(z), \qquad L_{P^\circ,f}(n) = \sum_{z \in nP^\circ \cap \ZZ^m} f(z).
\]
Then $L_{P,f}$ and $L_{P^\circ,f}$ extend to quasipolynomial functions which satisfy
\[
L_{-P^\circ, f}(x) = (-1)^{\dim P} L_{P,f}(-x).
\]
Furthermore, if $P$ is a lattice polytope, then $L_P$ and $L_{P^\circ}$ are polynomial.
\end{theorem}

Again, experts in Ehrhart theory will probably not find 
this result surprising or difficult to prove, but we have only
seen it stated explicitly in \cite{Ard09} and (without proof) in \cite{CJM11}.
We will only use it for $P = P_X(c) = \{z \in \RR^m \, : \,
Xz=c, z \geq 0\}$ and $f=\pi_Y$ for $Y \subseteq X$, where $\pi_Y(z)
= \prod_{i \in Y} z_i$, so we will present a proof for this
case. For a proof of the general statement,
see \cite{Ard09}.

\begin{proof}[Proof of Theorem \ref{th:Ehrhart} for $P = P_X(c)$ and $f=\pi_Y$] 
By (\ref{eq:pp})
 we have 
\begin{equation}\label{eq:Ehr1.5}
L_{P_X(c),\pi_Y}(n) = \P_{X,\pi_Y}(nc) = \sum_{T \subseteq Y} (-1)^{|Y-T|} \P_{X\cup T}(nc) =  \sum_{T \subseteq Y} (-1)^{|Y-T|} L_{\P_{X \cup T}(c)}(n)
\end{equation}

Now we need an ``interior" version of (\ref{eq:Ehr1.5}). Let $\P^\circ_X(c)$ denote the number of ways of expressing $c$ as a \textbf{positive} combination of vectors in $X$ that uses all vectors in $X$. This is the number of lattice points in the interior $P_X(c)^\circ$. 
Also let $\P^\circ_{X,f}(c)$ be the sum of $f(y)$ over all $y \in P_X(c)^\circ$

Note that, by the same argument we used to prove (\ref{eq:pp}), we get
\[
\P^\circ_{X \cup Y}(c) = \sum_{k \in \P_X(c)^\circ \cap \ZZ^m} \prod_{i \in Y} (k_i-1) = \sum_{T \subseteq Y} (-1)^{|Y-T|}\P^\circ_{X, \pi_T}(c)
\]
which, using the inclusion-exclusion formula gives
\begin{equation}  \label{eq:Ehr2}
L_{P_X(c)^\circ,\pi_Y}(n) = \P^\circ_{X,\pi_Y}(nc) = \sum_{T \subseteq Y} \P^\circ_{X \cup T}(nc) = \sum_{T \subseteq Y} L_{P_{X \cup T}(c)^\circ}(n).
\end{equation}
To relate (\ref{eq:Ehr1.5}) and (\ref{eq:Ehr2}), notice that Ehrhart reciprocity (Theorem \ref{th:Ehrh}) 
tells us that  
\[
L_{\P_{X \cup T}(c)^\circ}(n) = (-1)^{\dim P_X(c)+|T|}  L_{P_{X \cup T}(c)}(-n)
\]
since $\dim P_{X \cup T}(nc) = \dim P_X(c)+|T|$. Finally it remains to notice that
\[
L_{-P_{X \cup T}(c)^\circ,\pi_Y}(n) = (-1)^{|Y|} L_{P_{X \cup T}(c)^\circ,\pi_Y}(n)
\]
for all natural numbers $n$, and hence for all $n$. 
Combined with (\ref{eq:Ehr1.5}) and (\ref{eq:Ehr2}), this gives the desired result.
\end{proof}

\section{Proofs of Theorems \ref{th:main} and \ref{th:degree}}\label{sec:proof}

Recall the we encode the double Gromov-Witten invariants of the
Hirzebruch surfaces $\F_k$ in the  
function
\[
F_{a,k,g}^{n_1,n_2}(\xx,\yy) = N_g^{\alpha, \beta, \widetilde\alpha, \widetilde\beta}(a,b,k)
\]
where $\alpha_i,  \beta_i$ (resp. $\widetilde\alpha_i, \widetilde\beta_i)$ denote the number of entries of $\xx, \yy$ that are equal to $-i$ (resp. $i$), and $b = \sum i(\widetilde \alpha_i + \widetilde \beta_i)$.
We will need two simple lemmas. 

\begin{lemma}\label{lemma:genus}
The genus of a floor diagram $\D$ is given by $g(\D) = 1-v_B + v_G$, where $v_B$ and $v_G$ are the numbers of black and gray vertices, respectively.
\end{lemma}

\begin{proof}
The genus of $\D$ is
\[
g(\D) = 1-|V|+|E| = 1-(v_B+v_W+v_G) + (e_{BW}+e_{BG})
\]
where $v_B,v_W,v_G$ denote the number of black, white, and gray vertices, respectively, and $e_{BW}, e_{BG}$ denote the number of black-white and black-gray edges. Since every white vertex has degree $1$ we have $e_{BW}=v_W$. Since every gray vertex has degree $2$, we have $e_{BG} = 2v_G$. Therefore
\[
g(\D) = 1-v_B + v_G
\]
as desired.
\end{proof}

The following lemma is clear from the definitions and Lemma \ref{lemma:genus}.

\begin{lemma}\label{lemma:VEcount}
A floor diagram for $\F_k$ of bidegree $(a,b)$ and type 
$(n_1, n_2)$ 
has:

$\bullet$
$a$ black vertices, $g+a-1$ gray vertices, and $n_1 + n_2$ white vertices.

$\bullet$
$2(g+a-1)$ black-gray edges, and $n_1 + n_2$ black-white edges.
\end{lemma}

We are now ready to prove our main results.

\subsection{Proof of Theorem \ref{th:main}}

\begin{reptheorem}{th:main}
\textit{
Let $k,g, n_1, n_2 \ge 0$ and $a \ge 1$ be fixed  integers. 
The function 
\[
F_{a,k,g}^{n_1,n_2}(\xx,\yy)
\]
of double Gromov-Witten invariants of the Hirzebruch surface $\F_k$
is piecewise polynomial 
relative to
the chambers of the 
hyperplane arrangement
\begin{eqnarray*}
\sum_{i \in S} x_i + \sum_{j \in T} y_j + rk= 0 &&  
(S \subseteq [n_1], \,\, T\subseteq [n_2], \,\, 0 \leq r \leq a), \\
y_i-y_j=0 &&  (1 \leq i < j \leq n_2)
\end{eqnarray*}
inside 
$\Lambda=\{(x_1,\ldots,x_{n_1},y_1,\ldots,y_{n_2}) \in \ZZ^{n_1}\times \ZZ^{n_2}
\ | \ \sum x_i +\sum
y_i+ak=0\}\subset
\RR^{n_1}\times \RR^{n_2}.$
}
\end{reptheorem}

\begin{proof}
By Theorem \ref{th:BM}, $F_{a,k,g}(\xx, \yy) = F^{n_1,n_2}_{a,k,g}(\xx, \yy)$ is given by $\sum_\D
\mu(\D)$ as we sum over all floor diagrams $\D$ for $\F_k$ having
bidegree $(a,b)$, genus $g$,  divergence multiplicity vector $(\alpha,
\beta, \widetilde \alpha, \widetilde \beta)$, and left/right sequence $\xx$.
For each such floor diagram $\D$, let $\underline{\D}$ be the \emph{unweighted floor diagram} obtained by removing the weights of $\D$. 
We let the underlying graph $\underline{\D}$ inherit the partition $V = L \sqcup C \sqcup R$ of the vertices, the ordering of $C$, and the coloring of the vertices.
By Lemma \ref{lemma:VEcount}, the collection $\mathcal{G}$ of underlying graphs $\underline{\D}$ that may contribute to $F_{a,k,g}(\xx,\yy)$ is finite in number, and depends only on $g, a,$ and $n_1+n_2$. 

For each graph $G \in \mathcal{G}$, let $W_{G,k}(\xx, \yy)$ be the set of weightings $w: E(G) \rightarrow \ZZ_{>0}$ for which the resulting $\D$ is a floor diagram for $\F_k$ (so in particular every black vertex has divergence $k$ and every gray vertex has divergence $0$) whose white divergence sequence is $(\xx,\yy)$. Note that such a $\D$ automatically has genus $g$ and bidegree $(a,b)$.

The multiplicity of the resulting floor diagram $\D$ is 
$\pi_\inte(w)$, where $\pi_{\inte}: \RR^{E(G)} \rightarrow \RR$ is the
polynomial function defined by 
$\pi_{\inte}(w) = \displaystyle{\prod_{e \textrm{ internal}}} w(e)$. 
Therefore
\begin{equation}\label{eq:Gcontrib}
F_{G,k}(\xx, \yy) = \sum_{w \in W_{G,k}(\xx, \yy)} \pi_{\inte}(w)
\end{equation}
is a contribution of $G$ to $F_{a,k,g}(\xx, \yy)$; but it is not the only one. 
We need to keep in mind that $F_{G,k}(\xx, \yy)$ depends on the order of the entries of $\yy$, while in $F_{a,k,g}(\xx,\yy)$ we need to consider all the distinct orders for $\yy$;
see Theorem \ref{th:BM}. 

It follows that 
\begin{equation}\label{eq:FinGcontribs}
F_{a,k,g}(\xx, \yy) = \frac1{\beta_1! \beta _2!\cdots \widetilde\beta_1! \widetilde\beta_2! \cdots}\
\sum_{G \in \mathcal{G}} \ \sum_{\sigma \in \mathfrak S_{n_2}} F_{G,k}(\xx, \sigma(\yy))
\end{equation}
where $\mathfrak S_{n_2}$ is the set of permutations of a set with
$n_2$ elements.


Now let us study the function $F_{G,k}(\xx,\yy)$ of (\ref{eq:Gcontrib}) more closely.
Before we proceed with the general case, let us discuss the example of Figure  \ref{fig:floordiagramunlabelled}.

\begin{center}
\begin{figure}[h]
\begin{tabular}{ccc}
 \includegraphics[height=2cm]{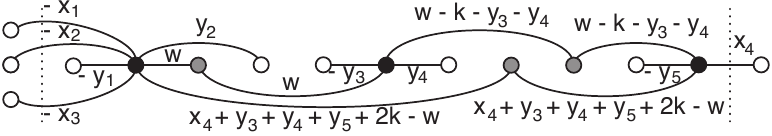}
\end{tabular}
\caption{A floor diagram for the Hirzebruch surface $\F_2$.}
\label{fig:floordiagramunlabelled}
\end{figure}
\end{center}

Consider a weighting giving rise to divergences $(\xx,\yy)$ at the white vertices, $k$ at the black vertices, and $0$ at the gray vertices. The weighting is fully determined by the weight $w$ of the edge from the first black to the first gray vertex, as shown. For this graph $G$ to contribute to 
$F_{a,k,g}(\xx,\yy)$
in the first place, we need
\[
x_1, x_2, x_3, y_1, y_3, y_5 < 0, \qquad x_4, y_2, y_4 >0.
\]
Also, for the weight $w$ to lead to a valid weighting, we need
\[
w>0, \qquad  w-k-y_3-y_4>0 \qquad x_4+y_3+y_4+y_5 + 2k-w>0.
\]
Therefore $F_{G,k}(\xx,\yy)$ equals
\[
 \sum_{w = \max(0, k+y_3+y_4)}^{x_4+y_3+y_4+y_5 + 2k} (-y_1)y_2(-y_3)y_4(-y_5)w^2 (w-k-y_3-y_4)^2 (x_4+y_3+y_4+y_5 + 2k-w)^2.
\]
If we fix the relative order of $0, y_3+y_4+k$ and $x_4+y_3+y_4+y_5 + 2k$, this function is clearly given by a fixed polynomial in $(\xx, \yy)$ for fixed $a$ and $k$. However, this polynomial changes when we change their relative order.

In the general case, the set of weightings $W_{G,k}(\xx,\yy)$ that interests us
is equal to the set of lattice points in a flow polytope. Given a sequence $\dd \in
\RR^V$, the \emph{flow polytope} $\Phi_G(\dd)$ 
is 
\[
\Phi_G(\dd) = \left\{w \in \RR^{E(G)} \, : \, w_e \geq 0 \textrm{ for all edges } e, \dive(e) = \dd_v \textrm{ for all vertices } v\right\},
\]
where we think of $w$ as a vector of flows on the edges of $G$ and, as before, the \emph{divergence} of a vertex is defined to be
\[
 \dive(v) = \sum_{ \tiny
     \begin{array}{c}
  \text{edges }e\\
v' \stackrel{e}{\to} v
     \end{array}
} w_e 
-   \sum_{ \tiny
     \begin{array}{c}
  \text{edges }e\\
v \stackrel{e}{\to} v'
     \end{array}
} w_e. 
\]
The flow polytope may be rewritten in matrix form as
\[
\Phi_G(\dd) = \{w \in \RR^{E(G)} \, : \, A \ww = \dd, \ww \geq 0\},
\]
where $A \in \RR^{V(G) \times E(G)}$ is the \emph{adjacency matrix} of $G$, defined by
\[
A_{v,e} = 
\begin{cases}
1 & \textrm{ when } v' \stackrel{e}{\to} v \textrm{ for some } v' \\
-1 & \textrm{ when } v \stackrel{e}{\to} v' \textrm{ for some } v' \\
0 & \textrm{ otherwise}.
\end{cases}
\]
Then clearly $W_{G,k}(\xx,\yy) = \Phi_G(\dd)$ where the entries of $\dd$ are given by $(\xx,\yy)$ for the white vertices, and are equal to $k$ for the black vertices and $0$ for the gray vertices.

From Theorem \ref{th:pp}, taking into account that the columns of the
adjacency matrix $A$ are a subset of the (unimodular) root system
$A_{|E(G)|-1}$, we obtain that the weighted partition function
\[
\P_{G, \pi_{\inte}}(\dd) = \sum_{w \in \Phi_G(\dd)} \pi_{\inte}(w)
\]
is piecewise polynomial
relative to
the chambers of the all-subset hyperplane arrangement in $\{\dd \in \RR^V \, : \,\sum_i d_i = 0\}$. Recall that this arrangement consists of the hyperplanes $\sum_{i \in V'} d_i = 0$ for all subsets $V' \subseteq V$.

We are only interested in the values of this function $\P_{G, \pi_{\inte}}(\dd)$ on the subspace  determined by the equations
\[
d_u = 0 \textrm{ for all gray } u, \qquad d_v = k \textrm{ for all black } v. 
\]
Since the sum of the divergences is $0$, we have $\sum x_i + \sum y_j +ak=0$ automatically.

The restriction of the weighted partition function $\P_{G, \pi_{\inte}}(\dd)$ to this subspace is the function  $F_{G,k}(\xx,\yy)$ of (\ref{eq:Gcontrib}). It remains piecewise polynomial, and the chamber structure is as stated.
When we symmetrize in (\ref{eq:FinGcontribs}), the result $\sum_{\sigma
  \in \mathfrak S_{n_2}} F_{G,k}(\xx, \sigma(\yy))$ is still piecewise
polynomial 
relative to
the same chambers, since the chamber structure is fixed under permutation of the $n_2$ $\yy$-variables. 
  The desired result follows.
\end{proof}

\subsection{Proof of Theorem \ref{th:degree}}

Having established the piecewise polynomiality of $F_{a,k,g}^{n_1,n_2}(\xx,\yy)$ by framing in terms of lattice point enumeration, we are ready to prove our next theorem. A similar argument was given in \cite{Ard09} and \cite{CJM11} for double Hurwitz numbers. 

\begin{reptheorem}{th:degree}
\textit{
Each polynomial piece of $F_{a,k,g}^{n_1,n_2}(\xx,\yy)$ has degree $n_2+3g+2a-2$, and is either even or odd.
}
\end{reptheorem}

\begin{proof}
In the notation of the proof of Theorem \ref{th:main}, it suffices to show these claims for the following piecewise polynomial function for each graph $G$:
\[
F_{G,k}(\xx,\yy) = \sum_{w \in W_{G,k}(\xx,\yy)} \pi_{\inte}(w)
\]

The degree of the polynomial $\pi_{\inte}(w)$ is the number of interior edges; by Lemma \ref{lemma:VEcount} this is $n_2 + 2(g+a-1)$. In each full-dimensional chamber, the polytope $W_{G,k}(\xx,\yy)$ has dimension $g$; to see this, observe that if we fix the flow on $g$ edges whose removal turns the graph into a tree, the whole flow vector will be uniquely determined. Clearly $g$ is the smallest number with this property.

Repeating the argument at the end of Theorem \ref{th:pp}, it follows that the polynomial pieces of $F_{G,k}(\xx,\yy)$ have degree $g + [ n_2 + 2(g+a-1)]$, 
which proves the second claim.


For the first claim, notice that
\[
F_{G,k}(t\xx, t\yy) = \sum_{w \in tW} \pi_{\inte}(w) = L_{W, \pi_{\inte}}(t)
\]
where we write $W=W_{G,k}(\xx,\yy)$. Therefore if $W^{\circ}$ denotes the relative interior of $W$,
\[
F_{G,k}(-t\xx, -t\yy) = L_{W, \pi_{\inte}}(-t) = (-1)^g L_{-W^\circ, \pi_{\inte}}(t) \\
\]
using weighted Ehrhart reciprocity (Theorem
\ref{th:Ehrhart}). Recalling that the number of internal edges in $G$
is always $i = n_2+2(g+a-1)$, we have 
$\pi_{\inte}(-w) = (-1)^i \pi_{\inte}(w)$ for any $w \in
\RR^E(G)$, 
so we get
\[
F_{G,k}(-t\xx, -t\yy) = (-1)^{g+i} L_{W^\circ, \pi_{\inte}}(t)
\]
and since 
$\pi_{\inte}(w) = 0$ whenever $w$ 
is in the boundary of $W^\circ$ (which is given by equalities of the
form 
$w_e = 0$),
 we get
\[
F_{G,k}(-t\xx, -t\yy) = (-1)^{g+i} L_{W, \pi_{\inte}}(t) = (-1)^{n_2+3g+2a-2} F_{G,k}(t\xx, t\yy).
\]
Therefore, depending on the parity of $n_2+3g+2a-2$, $F_{G,k}(t\xx, t\yy)$ is even or odd.
\end{proof}

\section{Example}\label{sec:example}

We conclude by computing explicitly the functions
$F_{2,k,g}^{2,1}(x_1,x_2,y_1)$ for any Hirzebruch surface $\F_k$ and
any genus $g$. They are listed in Table \ref{table:example}; see the last paragraph of this section
for the conventions used.  
The domain $\Lambda=\{(\xx,\yy) \in \ZZ^2 \times \ZZ \, : \,
x_1+x_2+y_1+2k=0\}$ is 
divided into 16 chambers, inside each one of which the function is polynomial. 
\begin{center}
\begin{figure}[h]
\begin{tabular}{ccc}
 \includegraphics[height=8cm]{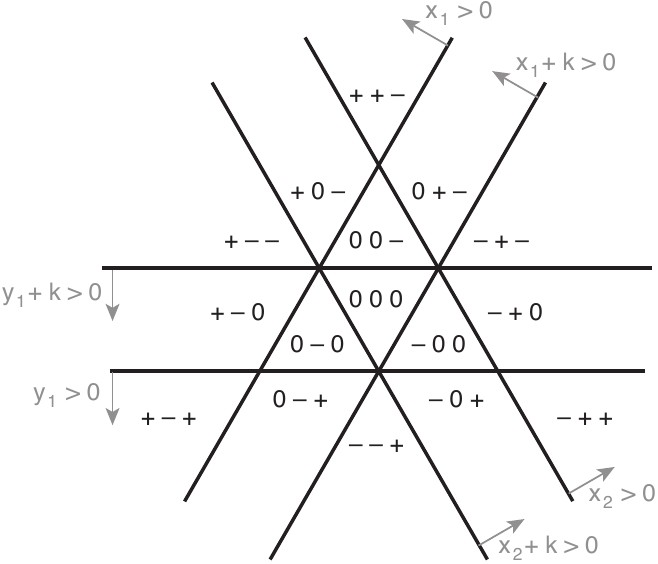}
\end{tabular}
\caption{The chamber complex for $F_{2,k,g}^{2,1}(x_1,x_2,y_1)$.}
\label{fig:chamber}
\end{figure}
\end{center}
These chambers are cut out by the six
planes with equations 
\[
x_1 =0, \quad x_1+k=0, \quad x_2 =0, \quad x_2+k=0, \quad y_1=0,\quad y_1+k=0.
\]
as shown in Figure \ref{fig:chamber}. We label each
 chamber with  a triple $s_1s_2s_3$, where each $s_i$ is $+$, $0$, or $-$ according to whether the corresponding variable is greater than $0$, between $-k$ and $0$, or less than $-k$, respectively. For example, the chamber $+\ 0\ -$ is given by the inequalities
\[
  x_1+k > x_1 > 0, \qquad
  x_2+k > 0 > x_2, \qquad
  0 > y_1+k > y_1.
\]

Since $F_{2,k,g}^{2,1}(x_2,x_1,y_1) = F_{2,k,g}^{2,1}(x_1,x_2,y_1)$,
it is sufficient to compute this function for $x_1 \ge x_2$. For this
reason, we focus on the $10$ chambers intersecting the half plane
$x_1\ge x_2$; the corresponding polynomials are listed in Table \ref{table:example}.
The polynomials on the remaining 6 chambers can be obtained by symmetry.

\medskip

We begin by discussing the case where the genus is $g=0$.
Figure
\ref{fig:graph} shows all graphs that can contribute to 
$F_{2,1,0}^{2,1}(x_1,x_2,y_1)$, 
obtained by 
 a careful but straightforward case-by-case analysis. 
  
\begin{center}
\begin{figure}[h]
\begin{tabular}{ccc}
 \includegraphics[width=5.5in]{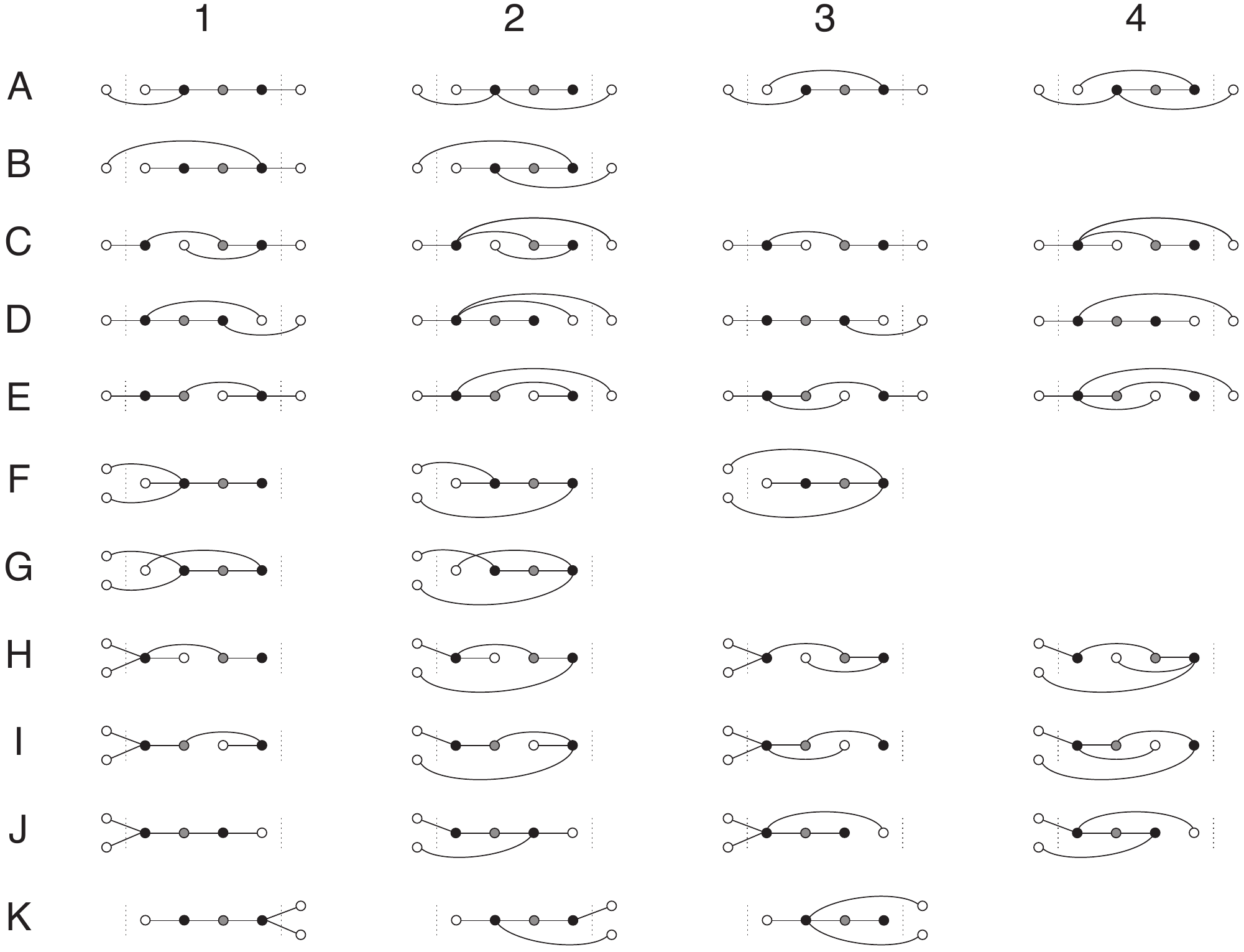}
\end{tabular}
\caption{The 38 graphs that contribute to $F_{2,1,0}^{2,1}(x_1,x_2,y_1)$.} 
\label{fig:graph}
\end{figure}
\end{center}

Each graph contributes in only some of the chambers. Consider, for example, graph $A1$. From left to right, its edge weights must be $-x_1, -y_1, x_2+k, x_2+k, x_2$ so that the vertices will have the correct divergences.
Therefore $A1$ contributes to $F_{2,1,0}^{2,1}(x_1,x_2,y_1)$ with weight $-y_1(x_2+k)^2$ as long as $x_1<0, y_1<0,$ and $x_2+k>0$; that is, in chambers $0+-$, $-+-$, and $-+0$.
Carrying out this computation for all graphs and chambers, we obtain
the polynomials of Table \ref{table:example} when $g=0$, 
with $\Gamma(w)=(w+k)^2$.

Note that for each graph in rows from $F$ to $K$ (i.e. when $x_1$ and
$x_2$ have the same sign),
 there are a priori two
different possibilities of labeling 
the vertices in $L$ or $R$ respectively with
$q_1$ and $q_2$, or $\widetilde q_1$ and $ \widetilde q_2$.
 The two corresponding floor diagrams are 
the same
for graphs
in columns $1$ and $3$, and are different for graphs in columns
$2$ and $4$ (even when $x_1=x_2$, since the corresponding two vertices in $L$ or $R$  are
not adjacent to the same vertex in $C$).

\begin{table}
\[
\begin{array}{|c|c|c|}
\hline
\textrm{Chamber} & \textrm{Graphs } (g=0) &  F_{2,k,g}^{2,1}(x_1,x_2,y_1)/|y_1|  
\\ 
\hline
0\ +- & $A1$, $A2$, $B1$, $B2$   &  \Gamma(x_1) +\Gamma(x_2) + \Gamma(y_1)
+ \Gamma(0)   \\
&&\\
\hline
-+-  &$A1$, $A2$, $A3$,  & (g+3)\Gamma(x_1) +\Gamma(x_2)
+\Gamma(y_1)+\Gamma(0)   
\\ & $B1$, $C1$, $E1$& 
\\
\hline
- + \ 0 &$A1$, $A2$, $A3$, $A4$& (g+3)\Gamma(x_1) +\Gamma(x_2)
+(g+3)\Gamma(y_1)+\Gamma(0) 
\\ &$C1$, $C2$, $E1$, $E2$  & 
\\
\hline
-++&$C3$, $C4$, $D1$, $D2$, &  \Gamma(x_1) +(g+3)\Gamma(x_2)+ \Gamma(y_1)
+(g+3)\Gamma(0)
\\ & $D3$, $D4$, $E3$, $E4$  &\\
\hline
-\ 0\ + & $H1$, $H2$, $I3$, $I4$  &  \Gamma(x_1) +(g+3)\Gamma(x_2)+ \Gamma(y_1)
  +(g+3)\Gamma(0)
\\ & $J1$, $J2$, $J3$, $J4$ &
\\
\hline
--+ & $H1$, $I3$,    & \Gamma(x_1) +\Gamma(x_2)+ \Gamma(y_1)  +(g+3)\Gamma(0) 
\\ &$J1$, 
$ J2$, 
$J2$,
$J3$ & \\
\hline
++- & $K1$, 
$K2$,
$K2$,
$K3$&\Gamma(x_1) +\Gamma(x_2)+  \Gamma(y_1)  +\Gamma(0)\\
&&\\
\hline
\ 0\ 0\ - & $F1$, 
$ F2$, 
$F2$,
$F3$& \Gamma(x_1) +\Gamma(x_2) +\Gamma(y_1)+\Gamma(0) 
\\
&&\\
\hline
-\ 0\   0 &$F1$, $F2$,  $G1$, $G2$, &(g+3)\Gamma(x_1) +\Gamma(x_2)
+(g+3)\Gamma(y_1)+\Gamma(0)
\\  &$H3$, $H4$, $I1$, $I2$  & 
\\
\hline
0\ 0\  0 &$F1$, 
$ F2$  
$F2$,
&\Gamma(x_1) +\Gamma(x_2) +(g+3)\Gamma(y_1)+\Gamma(0) 
\\ & $G1$, $H3$, $I1$  & \\
\hline
\end{array} 
\]
\caption{The double Gromov-Witten invariants $F_{2,k,g}^{2,1}(x_1,x_2,y_1)$.\label{table:example}}
\end{table}


\medskip

For higher genus $g$, the computation is not much more difficult in this special case.
In each graph we simply need to replace the gray vertex and its 2 incident edges by $g+1$ gray vertices and the corresponding $2(g+1)$ edges. When there is an intermediate white vertex, we simply need to decide 
its position among the $g+1$ gray vertices; there are $g+3$
choices. This gives rise to various factors of 
$g+3$
in Table \ref{table:example}.
For example, in chamber $-+-$ and genus $g=0$, the graphs $A3, C1, E1$
are isomorphic as unoriented
 graphs, and they account for the $3$ possible positions of the white vertex in $C$ relative to the black and gray vertices.

Suppose the two black-gray edges had weight $w$ in a graph of genus $0$. Now in the genus $g$ graph, that total weight $w$ has to be distributed among $g+1$ weights. Therefore the resulting contribution is 
\[
\Gamma_g(w)=\sum_{
w_1+\ldots +  w_{g+1}=w}
\,\,\, \prod_{i=1}^{g+1}w_i^2,
\]
where $w_1, \ldots, w_{g+1}$ are positive integers. Note that this is a polynomial of degree $3g+2$, which has the same parity as $g$.
For example we have 
\[
\Gamma_0(w)=w^2\quad \mbox{and}\quad \Gamma_1(w)=\frac{(w-1)w(w+1)(w^2+1)}{30}.
\]
To make Table \ref{table:example} easier to read, 
we divide $F_{w,k,g}^{2,1}(x_1,x_2,y_1)$ by 
$|y_1|$ and write
\[
\Gamma(w) = \Gamma_g(|w+k|).
\]

\section{Concluding remarks}\label{sec:conclusion}

The methods exposed in this note should also be useful in other related contexts:

\begin{itemize}
\item One can similarly define 
double Gromov-Witten invariants for more general toric
  surfaces. 
It should be possible to establish their piecewise polynomiality by pushing 
our method through, at least for toric surfaces corresponding to $h$-transverse polygons 
(see \cite{Br6b} for the definition of $h$-transversality).
Using methods similar to ours, \cite{ArdBlo} and \cite{LiuOss} prove the
  polynomiality of Severi degrees of many toric surfaces, including
  many singular ones. These papers show that Severi degrees also vary nicely as one
  changes the toric surface, and this might be the case for
  double Gromov-Witten invariants as well. (For instance, in Table \ref{table:example}, note that the polynomial in
  each chamber is also a 
 polynomial function of $k$). 

\item One may also try to extend this approach to 
\emph{double tropical Welschinger invariants}
of Hirzebruch surfaces (see \cite{IKS3} for the definition of tropical
Welschinger invariants). Due to the different treatment given to edges of
even and odd weights in the real multiplicity of a floor diagram,
there is no hope that double tropical Welschinger invariants are
piecewise polynomial. 
However it is reasonable to expect that they 
are piecewise quasipolynomial.

\item More generally, Block and G{\"o}ttsche defined in \cite{BlGo14}
 \emph{refined invariants} of toric surfaces. These invariants are
 univariate polynomials that
 interpolate between Gromov-Witten and tropical Welschinger
 invariants, and can also be computed via floor diagrams.
It would be interesting to apply the methods 
presented here to  double refined invariants.

\item It may also be possible to write explicit wall-crossing formulas
  describing how the function $F_{a,k,g}^{n_1,n_2}$ changes between two
  adjacent chambers. 
  For double Hurwitz numbers this was carried out
  in \cite{Ard09, CJM11}; it requires additional combinatorial
  insight and non-trivial technical hurdles. It would be interesting
  to extend it to this setting. 

\item 
Floor diagrams have higher dimensional versions, at least in genus 0
(see \cite{Br7,Br6}). However one should not expect analogous
piecewise polynomiality about curve enumeration in
spaces of dimension at least 3. Indeed, the multiplicity of a floor
diagram (i.e. the number of complex curves it encodes) includes the
multiplicity of its floors. A key point for piecewise polynomiality to
hold for double Hurwitz numbers and double Gromov-Witten invariants is that
in dimension 1 and 2, the multiplicity of
a floor is always equal to 1. However starting in dimension $n\ge 3$, the
multiplicity of a floor is itself a Gromov-Witten invariant of a space
of 
$n-1$,
and this invariant 
increases exponentially with the degree of the floor.

\section{Acknowledgments}
Federico Ardila would like to thank Erwan Brugall\'e and the Institut de Math\'ematiques de Jussieu for their hospitality; this work was carried out during a visit to Paris in early 2013 full of delicious food and wine, great music and company, and interesting mathematics. He also thanks San Francisco State University and the US National Science Foundation 
for their financial support and the University of California at Berkeley for hosting him during his sabbatical. Finally, we would like to thank the anonymous referees for their valuable feedback.

\end{itemize}

\bibliographystyle{alpha}
\bibliography{Biblio}

\begin{thebibliography}{ABLdM11}

\bibitem[AB13]{ArdBlo}
F.~Ardila and F.~Block.
\newblock Universal polynomials for {S}everi degrees of toric surfaces.
\newblock {\em Adv. Math.}, 237:165--193, 2013.

\bibitem[ABLdM11]{Br8}
A.~Arroyo, E.~Brugall\'e, and L.~Lopez~de Medrano.
\newblock Recursive formula for {W}elschinger invariants.
\newblock {\em Int Math Res Notices}, 5:1107--1134, 2011.

\bibitem[Ard09]{Ard09}
F.~Ardila.
\newblock Double {H}urwitz numbers and {DPV} remarkable spaces.
\newblock Preprint, 2009.

\bibitem[BBM14]{Br9}
B.~Bertrand, E.~Brugall{\'e}, and G.~Mikhalkin.
\newblock Genus 0 characteristic numbers of the tropical projective plane.
\newblock {\em Compos. Math.}, 150(1):46--104, 2014.

\bibitem[BCK13]{BlCoKe13}
F.~Block, S.~J. Colley, and G.~Kennedy.
\newblock Computing {S}everi degrees with long-edge graphs.
\newblock arXiv:1303.5308, 2013.

\bibitem[Bea83]{Beau}
A.~Beauville.
\newblock {\em Complex algebraic surfaces}, volume~68 of {\em London
  Mathematical Society Lecture Note Series}.
\newblock Cambridge University Press, Cambridge, 1983.

\bibitem[BG14a]{BlGo14bis}
F.~Block and L.~G{\"o}ttsche.
\newblock Fock spaces and refined {S}everi degrees.
\newblock arXiv:1409.4868, 2014.

\bibitem[BG14b]{BlGo14}
F.~Block and L.~G{\"o}ttsche.
\newblock Refined curve counting with tropical geometry.
\newblock arXiv:1407.2901, 2014.

\bibitem[BGM12]{BGM}
F.~Block, A.~Gathmann, and H.~Markwig.
\newblock Psi-floor diagrams and a {C}aporaso-{H}arris type recursion.
\newblock {\em Israel J. Math.}, 191(1):405--449, 2012.

\bibitem[Bla64]{Bla}
G.~R. Blakley.
\newblock Combinatorial remarks on partitions of a multipartite number.
\newblock {\em Duke Math. J.}, 31:335--340, 1964.

\bibitem[Blo11]{Blo11}
F.~Block.
\newblock Relative node polynomials for plane curves.
\newblock In {\em 23rd {I}nternational {C}onference on {F}ormal {P}ower
  {S}eries and {A}lgebraic {C}ombinatorics ({FPSAC} 2011)}, Discrete Math.
  Theor. Comput. Sci. Proc., AO, pages 199--210. Assoc. Discrete Math. Theor.
  Comput. Sci., Nancy, 2011.

\bibitem[BM]{Br6}
E.~Brugall\'e and G.~Mikhalkin.
\newblock Floor decompositions of tropical curves : the general case.
\newblock In preparation, preliminary version available at
  http://www.math.jussieu.fr/$\sim$brugalle/articles/FDn/FDGeneral.pdf.

\bibitem[BM07]{Br7}
E.~Brugall\'e and G.~Mikhalkin.
\newblock Enumeration of curves via floor diagrams.
\newblock {\em Comptes Rendus de l'Académie des Sciences de Paris}, série I,
  345(6):329--334, 2007.

\bibitem[BM08]{Br6b}
E.~Brugall\'e and G.~Mikhalkin.
\newblock Floor decompositions of tropical curves : the planar case.
\newblock {\em Proceedings of 15th {G}\"okova {G}eometry-{T}opology
  Conference}, pages 64--90, 2008.

\bibitem[BP13]{Br14}
E.~Brugall{\'e} and N.~Puignau.
\newblock Enumeration of real conics and maximal configurations.
\newblock {\em J. Eur. Math. Soc. (JEMS)}, 15(6):2139--2164, 2013.

\bibitem[BR07]{BecRob}
M.~Beck and S.~Robins.
\newblock {\em Computing the continuous discretely}.
\newblock Undergraduate Texts in Mathematics. Springer, New York, 2007.
\newblock Integer-point enumeration in polyhedra.

\bibitem[Bru14]{Bru14}
E.~Brugall\'e.
\newblock Floor diagrams of plane curves relative to a conic and {GW-W}
  invariants of {D}el {P}ezzo surfaces.
\newblock arXiv:1404.5429, 2014.

\bibitem[BV97]{BrVe}
M.~Brion and M.~Vergne.
\newblock Lattice points in simple polytopes.
\newblock {\em J. Amer. Math. Soc}, 10:371--392, 1997.

\bibitem[CJM10]{CJM}
R.~Cavalieri, Paul J., and H.~Markwig.
\newblock Tropical {H}urwitz numbers.
\newblock {\em J. Algebraic Combin.}, 32(2):241--265, 2010.

\bibitem[CJM11]{CJM11}
R.~Cavalieri, P.~Johnson, and H.~Markwig.
\newblock Wall crossings for double {H}urwitz numbers.
\newblock {\em Adv. Math.}, 228(4):1894--1937, 2011.

\bibitem[DCPV10a]{DPV1}
C.~De~Concini, C.~Procesi, and M.~Vergne.
\newblock Vector partition functions and generalized dahmen and micchelli
  spaces.
\newblock {\em Transformation Groups}, 15(4):751--773, 2010.

\bibitem[DCPV10b]{DPV2}
C.~De~Concini, C.~Procesi, and M.~Vergne.
\newblock Vector partition functions and index of transversally elliptic
  operators.
\newblock {\em Transformation Groups}, 15(4):775--811, 2010.

\bibitem[Ehr62]{Ehr}
E.~Ehrhart.
\newblock Sur les poly\`edres rationnels homoth\'etiques \`a {$n$}\ dimensions.
\newblock {\em C. R. Acad. Sci. Paris}, 254:616--618, 1962.

\bibitem[ELSV01]{ELSV}
T.~Ekedahl, S.~Lando, M.~Shapiro, and A.~Vainshtein.
\newblock Hurwitz numbers and intersections on moduli spaces of curves.
\newblock {\em Invent. Math.}, 146(2):297--327, 2001.

\bibitem[FH91]{FulHar}
W.~Fulton and J.~Harris.
\newblock {\em Representation theory}, volume 129 of {\em Graduate Texts in
  Mathematics}.
\newblock Springer-Verlag, New York, 1991.
\newblock A first course, Readings in Mathematics.

\bibitem[FM10]{FM}
S.~Fomin and G.~Mikhalkin.
\newblock Labelled floor diagrams for plane curves.
\newblock {\em Journal of the European Mathematical Society}, 12:1453--1496,
  2010.

\bibitem[GH94]{GriHar78}
P.~Griffiths and J.~Harris.
\newblock {\em Principles of algebraic geometry}.
\newblock Wiley Classics Library. John Wiley \& Sons, Inc., New York, 1994.
\newblock Reprint of the 1978 original.

\bibitem[GJV05]{Vak3}
I.~Goulden, D.~Jackson, and R.~Vakil.
\newblock Towards the geometry of double {H}urwitz numbers.
\newblock {\em Adv. Math.}, 198(1):43--92, 2005.

\bibitem[IKS09]{IKS3}
I.~Itenberg, V.~Kharlamov, and E.~Shustin.
\newblock A {C}aporaso-{H}arris type formula for {W}elschinger invariants of
  real toric {D}el {P}ezzo surfaces.
\newblock {\em Comment. Math. Helv.}, 84:87--126, 2009.

\bibitem[IP04]{IP00}
E.-N Ionel and T.~H. Parker.
\newblock The symplectic sum formula for {G}romov-{W}itten invariants.
\newblock {\em Ann. of Math.}, 159(2):935--1025, 2004.

\bibitem[Kos59]{Kos}
B.~Kostant.
\newblock A formula for the multiplicity of a weight.
\newblock {\em Trans. Amer. Math. Soc.}, 93:53--73, 1959.

\bibitem[Li02]{Li02}
J.~Li.
\newblock A degeneration formula of {GW}-invariants.
\newblock {\em J. Differential Geom.}, 60(2):199--293, 2002.

\bibitem[Liu13]{Liu13}
F.~Liu.
\newblock A combinatorial analysis of {S}everi degrees.
\newblock arXiv:1304.1256, 2013.

\bibitem[LO14]{LiuOss}
F.~Liu and B.~Osserman.
\newblock Severi degrees on toric surfaces.
\newblock arXiv:1401.7023, 2014.

\bibitem[LR01]{LiRu01}
A.~Li and Y.~Ruan.
\newblock Symplectic surgery and {G}romov-{W}itten invariants of {C}alabi-{Y}au
  3-folds.
\newblock {\em Invent. Math.}, 145(1):151--218, 2001.

\bibitem[Mik05]{Mik1}
G.~Mikhalkin.
\newblock {Enumerative tropical algebraic geometry in $\mathbb R^2$}.
\newblock {\em J. Amer. Math. Soc.}, 18(2):313--377, 2005.

\bibitem[Shu12]{Shu9}
E.~Shustin.
\newblock Tropical and algebraic curves with multiple points.
\newblock In {\em Perspectives in analysis, geometry, and topology}, volume 296
  of {\em Progr. Math.}, pages 431--464. Birkh\"auser/Springer, New York, 2012.

\bibitem[SSV08]{SSV}
S.~Shadrin, M.~Shapiro, and A.~Vainshtein.
\newblock Chamber behavior of double {H}urwitz numbers in genus 0.
\newblock {\em Adv. Math.}, 217(1):79--96, 2008.

\bibitem[Stu95]{Stu95}
B.~Sturmfels.
\newblock On vector partition functions.
\newblock {\em J. Combin. Theory Ser. A}, 72(2):302--309, 1995.

\bibitem[Vak00]{Vak2}
R.~Vakil.
\newblock Counting curves on rational surfaces.
\newblock {\em Manuscripta math.}, 102:53--84, 2000.

\end{thebibliography}

\end{document}